\documentclass[reqno]{amsart}

\usepackage{amsmath,amsfonts,amssymb}

\newtheorem{thm}{Theorem}

\newtheorem{lem}{Lemma}

\newtheorem{definition}{Definition}

\begin{document}
\title[New Integral Inequalities]{New Integral Inequalities through Generalized Convex Functions with Application}
\author[M. Muddassar]{Muhammad Muddassar}
\address{Department of Mathematics, University of Engineering and Technology, Taxila, Pakistan}
\email{malik.muddassar@gmail.com}
\author[A. Ali]{Ahsan Ali}
\address{Department of Electronic Engineering, University of Engineering and Technology, Taxila, Pakistan}
\email{ahsan.ali@uettaxila.edu.pk}
\date{\today}

\begin{abstract}
In this paper, we establish various inequalities for some mappings that are linked with the illustrious Hermite-Hadamard integral inequality for mappings whose absolute values belong to the class ${K_{m, 1}^{\alpha, s}}$ and ${K_{m, 2}^{\alpha, s}}$.
\end{abstract}
\maketitle
\section{Introduction}\label{Sec 1}
The role of mathematical inequalities within the mathematical branches as well as their enormous applications can not be underestimated. The appearance of the new mathematical inequality often puts on firm foundation for the heuristic algorithms and procedures used in applied sciences. Among others, one of the main inequality which gives us an explicit error bounds in the trapezoidal and midpoint rules of a smooth function, called Hermit-Hadamard's inequality, is defined as ~\cite[p. 53]{r9}:
\begin{equation}\label{1}
    f\left(\frac{a+b}{2}\right)\,\leq\,\frac{1}{b-a}\int_a^b\,f(x)\,dx\,\leq\ \frac{f(a)+f(b)}{2},
\end{equation}
where $f:[a,b]\rightarrow\mathbb{R}$ is a convex function. Both inequalities hold in the reversed direction for $f$ to be concave.\\
We note that Hermit-Hadamard's inequality  $(\ref{1})$ may be regarded as a refinement of the concept of convexity and it follows easily from Jensen's inequality. Inequality $(\ref{1})$ has received renewed attention in recent years and a variety of refinements and generalizations can be found in many Articles, Books, Volumes and Journals.
The notion of quasi-convex functions generalizes the notion of convex functions. More precisely, a function $f : [a, b] \rightarrow \mathbb{R}$ is said to be quasi-convex on $[a, b]$ if
$$f(tx+(1-t)y) \leq \max\{f(x), f(y)\}$$
holds for any $x, y \in [a, b]$ and $t \in [0, 1]$. Clearly, any convex function is a quasi-convex function. Furthermore, there exist quasi-convex functions which are not convex (see ~\cite{r4a}).
In ~\cite{r8}, \"{O}zdemir et al. established several integral inequalities respecting some kinds of convexity. Especially, they discussed the following result connecting with quasi-convex functions:
\begin{thm}\label{T1}
Let $f : [a, b] \rightarrow \mathbb{R}$ be continuous on $[a, b]$ such that $f \in L([a, b])$, with $0 \leq a < b < \infty$. If $f$ is quasi-convex on $[a, b]$, then for some fixed $p, q > 0$, we have
\begin{equation*}
\int_a^b (x-a)^p(b-x)^q f(x)dx = (b-a)^{p+q+1} \beta(p+1, q+1) \max\{f(a), f(b)\}
\end{equation*}
where $\beta(x, y)$ is the Euler Beta function.
\end{thm}
Recently, Liu in ~\cite{r6}  established some new integral inequalities for quasiconvex functions as follows:
\begin{thm}\label{T2}
Let $f : [a, b] \rightarrow  \mathbb{R}$ be continuous on $[a, b]$ such that $f \in L([a, b])$, with $0 \leq a < b < \infty$ and let $k > 1$. If $|f|^{\frac{k}{k-1}}$ is quasi-convex on $[a, b]$, for some fixed $p, q > 0$, then
\begin{eqnarray*}
&&\!\!\!\!\!\!\!\!\!\!\!\!\!\!\!\! \int_a^b (x-a)^p(b-x)^q f(x)dx = (b-a)^{p+q+1} \left(\beta(kp+1, kq+1)\right)^{\frac{1}{k}}\times\\
&& \indent\indent\indent\indent\indent\indent\indent\indent\indent\indent\indent\indent\left(\max\{\left|f(a)\right|^{\frac{k}{k-1}}, \left|f(b)\right|^{\frac{k}{k-1}}\}\right)^{\frac{k-1}{k}}
\end{eqnarray*}
\end{thm}
\begin{thm}\label{T3}
Let $f : [a, b] \rightarrow  \mathbb{R}$ be continuous on $[a, b]$ such that $f \in L([a, b])$, with $0 \leq a < b < \infty$ and let $l \geq 1$. If $|f|^l$ is quasi-convex on $[a, b]$, for some fixed $p, q > 0$, then
\begin{equation*}
\int_a^b (x-a)^p(b-x)^q f(x)dx = (b-a)^{p+q+1} \beta(p+1, q+1) \left(\max\{\left|f(a)\right|^l, \left|f(b)\right|^l\}\right)^{\frac{1}{l}}
\end{equation*}
\end{thm}
That is, this study is a further continuation and generalization of ~\cite{r6a} and ~\cite{r8} via generalized convexity.
\section{Main Results}\label{Sec 1}
In this section, we generalize the above theorems and produced some more results using the following lemma described in ~\cite{r10}.
\begin{lem}\label{l1}
Let $f : [a, b] \subset [0,\infty) \rightarrow \mathbb{R}$ be continuous on $[a, b]$ such that $f \in L([a, b])$, with $a < b$. Then the equality
\begin{equation}\label{le1}
    \int_a^b (x-a)^p(b-x)^q f(x)dx = (b-a)^{p+q+1} \int_0^1 (1-t)^p t^q f(ta + (1-t)b)dt
\end{equation}
holds for some fixed $p, q > 0$.
\end{lem}
To prove our main results, we follow the following definitions first defined in ~\cite{r8} by M. Muddassar et. al., named as $s-(\alpha, m)$-convex functions as reproduced below;
\begin{definition}
A function $f : [0, \infty) \rightarrow [0, \infty)$ is said to be $s-(\alpha, m)$-convex function in the first sense or f belongs to the class ${K_{m, 1}^{\alpha, s}}$ , if for all $x, y \in [0, \infty)$ and $\mu \in  [0, 1]$, the following inequality holds:
\begin{equation*}
    f(\mu x + (1-\mu) y) \leq \left({\mu^\alpha}^s \right) f(x) + m \left(1-{\mu^\alpha}^s \right) f\left(\frac{y}{m}\right)
\end{equation*}
where $(\alpha,m) \in [0,1]^2$ and for some fixed $s \in (0, 1]$.
\end{definition}
\begin{definition}
A function $f : [0, \infty) \rightarrow [0, \infty)$ is said to be $s-(\alpha, m)$-convex function in the second sense or f belongs to the class ${K_{m, 2}^{\alpha, s}}$ , if for all $x, y \in [0, \infty)$ and $\mu, \nu \in  [0, 1]$, the following inequality holds:
\begin{equation*}
    f(\mu x + (1-\mu) y) \leq \left(\mu^\alpha \right)^s f(x) + m \left(1-\mu^\alpha \right)^s f\left(\frac{y}{m}\right)
\end{equation*}
where $(\alpha,m) \in [0,1]^2$ and for some fixed $s \in (0, 1]$.
\end{definition}
Note that for $s=1$, we get $K_m^\alpha(I)$ class of convex functions and for $\alpha=1$ and $m=1$, we get $K_s^1(I)$ and $K_s^2(I)$ class of convex functions.
\begin{thm}\label{th1}
Let $f : [a, b] \rightarrow \mathbb{R}$ be continuous on $[a, b]$ such that $f \in L([a, b])$, with $0 \leq a < b < \infty$. If $|f|$ belongs to the class ${K_{m, 1}^{\alpha, s}}$ on $[a, b]$, then for some fixed $p, q > 0$, we have
\begin{eqnarray}\label{te1}
&&\nonumber\!\!\!\!\!\!\!\!\!\!\!\!\!\!\!\int_a^b \!\!(x-a)^p(b-x)^q f(x)dx\! \leq \! \left(b-a\right)^{p+q+1}\!\left\{\beta(q+\alpha s+1,p+1)\!\left(\!\left|f(a)\right|\!-\!m\left|f\left(\!\frac{b}{m}\!\right)\right|\!\right)\right.\\
&& \indent\indent\indent\indent\indent\indent\indent\indent\indent\indent\indent\indent\indent\indent\indent \left. +m\beta(q+1, p+1)\left|f\left(\frac{b}{m}\right)\right|\right\}
\end{eqnarray}
\end{thm}
\begin{proof}
Taking absolute value of Lemma \ref{l1}, we have
\begin{equation}\label{t1a}
    \int_a^b (x-a)^p(b-x)^q f(x)dx \leq (b-a)^{p+q+1} \int_0^1 (1-t)^p t^q \left|f(ta + (1-t)b)\right|dt
\end{equation}
Since $|f|$ belongs to the class ${K_{m, 1}^{\alpha, s}}$ on $[a, b]$, then the integral on the right side of inequality (\ref{t1a}) can be written as
\begin{equation}\label{t1b}
    \int_0^1 (1-t)^p t^q \left|f(ta + (1-t)b)\right|dt \leq \int_0^1 (1-t)^p t^q \left(t^{\alpha s}|f(a)| + m(1-t^{\alpha s})|f(b)|\right)dt
\end{equation}
Now Here,
\begin{equation}\label{t1c}
    \int_0^1 (1-t)^p t^{q+\alpha s} dt= \beta(q+\alpha s+1, p+1)
\end{equation}
and
\begin{equation}\label{t1d}
    \int_0^1 (1-t)^p t^q (1-t^{\alpha s}) dt= \beta(q+1, p+1) - \beta(q+\alpha s+1, p+1)
\end{equation}
Using (\ref{t1b}), (\ref{t1c}) and (\ref{t1d}) in (\ref{t1a}) and doing some algebraic operations, we get (\ref{te1}).
\end{proof}
\begin{thm}\label{th2}
Let $f : [a, b] \rightarrow \mathbb{R}$ be continuous on $[a, b]$ such that $f \in L([a, b])$, with $0 \leq a < b < \infty$ and let $k > 1$. If $|f|^{\frac{k}{k-1}}$ belongs to the class ${K_{m, 1}^{\alpha, s}}$ on $[a, b]$, then for some fixed $p, q > 0$, we have
\begin{eqnarray}\label{te2}
&&\nonumber\!\!\!\!\!\!\!\!\!\!\!\!\!\!\!\int_a^b (x-a)^p(b-x)^q f(x)dx \leq  \left(b-a\right)^{p+q+1} \left(\beta(\alpha s+1, 1)\right)^{\frac{k-1}{k}}\left(\beta(qk+1, pk+1)\right)^{\frac{1}{k}} \\
&&\indent\indent\indent\indent\indent\indent\indent\indent\indent\indent\indent\indent\indent\indent\indent \left[\left|f(a)\right|^{\frac{k}{k-1}}+m\left|f\left(\frac{b}{m}\right)\right|^{\frac{k}{k-1}}\right]^{\frac{k-1}{k}}
\end{eqnarray}
\end{thm}
\begin{proof}
Applying the H\"{o}lder's Integral Inequality on the integral on the rightside of (\ref{t1a}), we get
\begin{eqnarray}\label{t2a}
&&\nonumber\!\!\!\!\!\!\!\!\!\!\!\!\!\!\!\int_0^1 (1-t)^p t^q \left|f(ta + (1-t)b)\right|dt \leq \left[\int_0^1 \left((1-t)^p t^q\right)^k dt\right]^{\frac{1}{k}}\times\\
&&\indent\indent\indent\indent\indent\indent\indent\indent\indent\indent\indent\indent \left[\int_0^1 \left|f(ta + (1-t)b)\right|^{\frac{k}{k-1}} dt\right]^{1-\frac{1}{k}}
\end{eqnarray}
Here,
\begin{equation}\label{t2b}
    \int_0^1 (1-t)^{pk} t^{qk} dt= \beta(qk+1, pk+1)
\end{equation}
Since $|f|^{\frac{k}{k-1}}$ belongs to the class ${K_{m, 1}^{\alpha, s}}$ on $[a, b]$ for $k > 1$, we have
\begin{equation}\label{t2c}
    \int_0^1 \left|f(ta + (1-t)b)\right|^{\frac{k}{k-1}} dt \leq \int_0^1 \left(t^{\alpha s}|f(a)|^{\frac{k}{k-1}} + m(1-t^{\alpha s})|f(b)|^{\frac{k}{k-1}} \right) dt
\end{equation}
furthermore,
\begin{equation}\label{t2d}
    \int_0^1 t^{\alpha s}dt=\int_0^1 (1-t)^{\alpha s}dt= \beta(\alpha s+1, 1)
\end{equation}
Inequalities (\ref{t1a}), (\ref{t2a}), (\ref{t2c}) and Equations (\ref{t2b}),(\ref{t2d}) together implies (\ref{te2}).
\end{proof}
\begin{thm}\label{th3}
Let $f : [a, b] \rightarrow \mathbb{R}$ be continuous on $[a, b]$ such that $f \in L([a, b])$, with $0 \leq a < b < \infty$ and let $l \geq 1$. If $|f|^l$ belongs to the class ${K_{m, 1}^{\alpha, s}}$ on $[a, b]$, then for some fixed $p, q > 0$, we have
\begin{eqnarray}\label{te3}
&&\nonumber\!\!\!\!\!\!\!\!\!\!\!\!\!\!\!\int_a^b (x-a)^p(b-x)^q f(x)dx \leq  \left(b-a\right)^{p+q+1} \left(\beta(q+1, p+1)\right)^{\frac{l-1}{l}}\left[\beta (q+\alpha s+1, p+1) \right.\\
&&\indent\indent\indent\indent\indent\indent\indent\indent \left.\left\{\!\left|f(a)\right|^l\!-\!m\left|\!f\left(\!\frac{b}{m}\!\right)\!\right|^l\!\right\}\!+m\beta(q\!+\!1, p\!+\!1)\!\left|\!f\left(\!\frac{b}{m}\!\right)\!\right|^l\!\right]^{\frac{1}{l}}
\end{eqnarray}
\end{thm}
\begin{proof}
Now applying the H\"{o}lder's Integral Inequality on the integral on the rightside of (\ref{t1a}) in anotherway, we get
\begin{eqnarray}\label{t3a}
&&\nonumber\!\!\!\!\!\!\!\!\!\!\!\!\!\!\!\int_0^1 (1-t)^p t^q \left|f(ta + (1-t)b)\right|dt \leq \left[\int_0^1 (1-t)^p t^q dt\right]^{1-\frac{1}{l}}\times\\
&&\indent\indent\indent\indent\indent\indent\indent\indent\indent\indent\indent\indent \left[\int_0^1 (1-t)^p t^q \left|f(ta + (1-t)b)\right|^l dt\right]^{\frac{1}{l}}
\end{eqnarray}
Here,
\begin{equation}\label{t3b}
    \int_0^1 (1-t)^p t^q dt= \beta(q+1, p+1)
\end{equation}
Since $|f|^l$ belongs to the class ${K_{m, 1}^{\alpha, s}}$ on $[a, b]$ for $l \geq 1$, we have
\begin{equation}\label{t2c}
    \!\!\!\! \int_0^1\!\! (1-t)^p t^q\! \left|f(ta\! +\! (1\!-\!t)b)\right|^l\! dt\!\leq \!\int_0^1\!\!\! (1\!-\!t)^p t^q \left(t^{\alpha s}|f(a)|^l\! +\! m(1-t^{\alpha s})|f(b)|^l \right)\! dt
\end{equation}
which completes the proof.
\end{proof}\\
Some more integral inequalities can be found using ${K_{m, 2}^{\alpha, s}}$ class of convex functions.

\end{document}